\def\PREP{PREP}
\def\ISSAC{ISSAC}
\def\form{PREP}

\ifx\form\PREP
  \documentclass[12pt]{amsart}
  
\fi

\ifx\form\ISSAC
  \documentclass{sig-alternate}

\fi

\usepackage[utf8]{inputenc}
\usepackage{amssymb}
\usepackage{amsmath}
\usepackage{amsthm} 
\usepackage{graphicx}
\usepackage{verbatim}
\usepackage{fancyhdr}
\usepackage{hyperref}
\usepackage{tikz-cd}
\usepackage{tikz}
\usepackage{algorithmic}
\usepackage{xypic}
\usetikzlibrary{arrows}

\theoremstyle{plain}
\newtheorem{theorem}{Theorem}[section]

\newtheorem{lemma}[theorem]{Lemma}
\newtheorem{proposition}[theorem]{Proposition}

\theoremstyle{definition}
\newtheorem{definition}[theorem]{Definition}

\newtheorem{example}[theorem]{Example}
\newtheorem{algorithm}[theorem]{Algorithm}

\numberwithin{equation}{section}

\newcommand{\B}[1]{\mathbb #1}
\newcommand{\C}[1]{\mathcal #1}
\newcommand{\F}[1]{\mathfrak #1}

\newcommand{\alg}[1]{\operatorname{#1}}
\newcommand{\ideal}[1]{\langle #1 \rangle}
\newcommand{\norm}[1]{\| #1 \|_1}

\DeclareMathOperator{\initial}{in}
\DeclareMathOperator{\NF}{NF}
\DeclareMathOperator{\lcm}{lcm}
\DeclareMathOperator{\im}{im}

\newcommand{\Inc}{\operatorname{Inc}(\B N)}
\newcommand{\SymN}{\F S_\infty}

\newcommand{\mon}{M}
\newcommand{\Sym}{\F S_\infty}
\newcommand{\LT}{\initial_{\leq}}
\newcommand{\LTp}{\initial_{\leq'}}
\newcommand{\LC}{\operatorname{LC}}

\newcommand{\mm}{\C M}
\newcommand{\degtot}{\operatorname{deg}_{tot}}

\begin{document}
\title{Equivariant Gr\"obner bases of symmetric toric ideals}
\ifx\form\ISSAC
\numberofauthors{1}
\author{
\alignauthor Robert Krone\\
       \affaddr{Queen's University, Kingston, Ontario.}
       \email{rk71@queensu.ca}
}
\fi

\ifx\form\PREP
\author{Robert Krone}
\address{Queen's University\\ Kingston, Ontario} 
\email{rk71@queensu.ca}
\fi

\maketitle

\begin{abstract}
 It has been shown previously that a large class of monomial maps equivariant under the action of an infinite symmetric group have finitely generated kernels up to the symmetric action.  We prove that these symmetric toric ideals also have finite Gr\"obner bases up to symmetry for certain monomial orders.  An algorithm is presented for computing equivariant Gr\"obner bases that terminates whenever a finite basis exists, improving on previous algorithms that only guaranteed termination in rings Noetherian up to symmetry.  This algorithm can be used to compute equivariant Gr\"obner bases of the above toric ideals, given the monomial map.
\end{abstract}

\section{Introduction}
Kernels of monomial maps between polynomial rings (i.e. toric ideals) can be used to describe certain statistical models \cite{diaconis1998algebraic}.  Often these maps coming from statistics have symmetry under the action of large symmetric groups, which has led to the study of toric ideals invariant under $\Sym$-action.  In various cases of interest these symmetric toric ideals were proved to be generated by the orbits of a finite number of binomials, even as they live in non-Noetherian rings e.g. \cite{aoki05:_markov_monte_carlo}\cite{hillar2012finite}.

A negative result \cite{loera06:_markov_bases_of_three_way} demonstrated monomial maps with kernels not finitely generated up to symmetry.  However in the case that the target ring has variables with at most one index running to infinity, \cite{draisma2013noetherianity} proved finite generation up to symmetry of the kernel always holds, generalizing the previous finite generation results.  This work left open many computational questions around efficiently finding generating sets up to symmetry of these ideal.  In some simple cases \cite{KKL:equivariant-markov} gave formulas for small generating sets.

The approach we take in this paper to computing generators is through Gr\"obner bases.  We prove that the symmetric toric ideals considered in \cite{draisma2013noetherianity} also have finite equivariant Gr\"obner bases, at least for specifically chosen monomial orders.  The question of whether finite Gr\"obner bases exist for all appropriate orders remains open.  Additionally we present an algorithm for computing equivariant Gr\"obner bases that has guaranteed termination as long as a finite Gr\"obner basis exists.  Although this algorithm concept has been used previously, it has not appeared before in the literature.  Previously presented algorithms guaranteed termination only in rings that are Noetherian up to symmetry, while many of the rings we work with fail this criterion.  This newly presented algorithm can compute a Gr\"obner basis of a symmetric toric ideal, given a description of the monomial map that defines it.

\section{Preliminaries}
Let $R$ be a graded commutative $K$-algebra equipped with a left action of monoid $\Pi$.  We mainly consider the case where $R$ has the structure of a monoid algebra, that is for some abelian monoid $\mon$, the elements of $R$ consist of formal sums of elements of $\mon$ with coefficients in $K$.  A common example of a graded monoid algebra is a polynomial ring $R = K[X]$ with variables from the set $X$ each with positive degree.  In this case $\mon$ is the free abelian monoid generated by $X$, which we will denote by $[X]$.  To make notation consistent with the polynomial case, we will denote the monoid algebra of $\mon$ over $K$ by $K\mon$.  We will also refer to elements of $\mon$ as ``monomials'' in analogy to the polynomial case.  Additionally the action of $\Pi$ on $R = K\mon$ will typically be induced by a $\Pi$-action on $\mon$ by monoid homomorphisms.

Our particular focus in this paper is when $\Pi$ is $\SymN$ or certain related monoids.  Here $\SymN$ will be defined as the group of all permutations of $\B N$ that fix all but a finite number of elements, $\SymN := \bigcup_n \F S_n$.

\begin{example}
 Let $R = K[x_1,x_2,x_3,\ldots]$ with $\SymN$ acting on $R$ by permuting the variables, so that $\sigma x_i = x_{\sigma(i)}$.
\end{example}

\begin{definition}
 An ideal $I \subseteq R$ is a {\em $\Pi$-invariant ideal} if $\sigma I \subseteq I$ for all $\sigma \in \Pi$.
\end{definition}

The ring $R$ is both an $R$-module and a $\Pi$-module, and there is a ring $R*\Pi$ which captures both of these actions, and which will be referred to as the {\em twisted monoid ring} of $\Pi$ with coefficients in $R$.  The elements of $R*\Pi$ are of the form $\sum_{\sigma \in \Pi} f_{\sigma}\cdot \sigma$ with each $f_\sigma \in R$ and only a finite number non-zero.  The additive structure is the same as the usual monoid ring, but multiplication is ``twisted'' in that
 \[ (f\cdot \sigma)(g \cdot \tau) = f\sigma(g) \cdot \sigma\tau \]
where $\sigma(g)$ denotes the element of $R$ obtained by acting on $g$ by $\sigma$.

$R$ is a $R*\Pi$-module, and the definition of $\Pi$-invariant ideals can be restated as $R*\Pi$-submodules of $R$.

When $R = K\mon$ with $\Pi$ acting on $\mon$, then we can define monoid $\mon *\Pi$ whose elements are pairs in $\mon \times \Pi$ with monoid operation
 \[ (m, \sigma)(n, \tau) = (m\sigma(n), \sigma\tau). \]
There is a left action of $\mon*\Pi$ on $\mon$ and the elements of $\mon *\Pi$ are the ``monomials'' of $R*\Pi$.

Some authors call a functor from category $\C C$ to the category of commutative $K$-algebras (or $K$-modules) a {\em $\C C$-algebra} (or {\em $\C C$-module} respectively) \cite{church2012fi}.  The monoid $\Pi$ can be considered as a category with single object $x$.  Algebra $R$ with left $\Pi$-action encodes the same data as a $\Pi$-algebra $F$ with $R = F(x)$.  We will therefore refer to $R$ as a $\Pi$-algebra, and a $\Pi$-invariant ideal is an ideal that is a $\Pi$-submodule of $R$.

\begin{definition}
 A $\Pi$-invariant ideal $I \subseteq R$ is {\em $\Pi$-finitely generated} if there is a finite set $F \subseteq I$ such that the $\Pi$-orbits of the elements of $F$ generate $I$.  Ring $R$ is called {\em $\Pi$-Noetherian} if every $\Pi$-invariant ideal in $R$ is $\Pi$-finitely generated.
\end{definition}
If a $\Pi$-invariant ideal $I$ is generated by the $\Pi$-orbits of a set $F$, we will write
 \[ I = \ideal{F}_{\Pi}. \]
Such a set $F$ generates $I$ as an $R*\Pi$-module.

A $\Pi$-algebra is $\Pi$-finitely generated if it is generated as an algebra by the $\Pi$-orbits of a finite set.

\begin{example}
 Continuing the above example of $R = K[x_1,x_2,x_3,\ldots]$ with $\SymN$ action, the ideal $\F m = \ideal{x_1,x_2,x_3,\ldots}$ is a $\SymN$-invariant ideal.  Moreover, it is $\SymN$-finitely generated because $\F m = \ideal{x_1}_{\SymN}$.  Also $R$ is a $\SymN$-finitely generated $K$-algebra with generator $x_1$.
\end{example}

\begin{definition}
 Let $R$ be a $\SymN$-algebra.  For $f \in R$, the {\em width} of $f$ is the smallest integer $n$ such that for every $\sigma \in \SymN$ that fixes $\{1,\ldots,n\}$, $\sigma$ also fixes $f$.  The width of $f$ is denoted $w(f)$.  If no such integer $n$ exists, then $w(f) = \infty$.  For a set $F \subseteq R$, its width is $w(F) := \sup_{f \in F}\{w(f)\}$.
\end{definition}
$\Sym$-algebra $R$ has the {\em finite width property} if every element of $R$ has finite width. The tools described in this paper generally apply only to such algebras, so from here forward every $\Sym$-algebra is assumed to have the finite width property.  For a $\SymN$-invariant ideal $I \subseteq R$ and an integer $n$ we can define the $n$th truncation of $I$ as
 \[ I_n := \{ f \in I \mid w(f) \leq n \}. \]
The set $I_n$ is naturally a $\F S_n$-invariant ideal of $R_n$.  An equivalent statement of the finite width property is that $R$ is equal to the direct limit of its truncations.  If $I$ is $\SymN$-finitely generated, then there is sufficiently large $n \in \B N$ such that $I = \ideal{I_n}_{\SymN}$.

The definition of width also applies to $\Pi = \Inc$, the monoid of strictly increasing functions, which is introduced below.
 
\begin{definition}
 Given $R = K\mon$ with $\Pi$ acting on $\mon$, there is a natural quasi-order $|_\Pi$ on $\mon$ called the {\em $\Pi$-divisibility quasi-order} defined by $a |_\Pi b$ if there exists $\sigma \in \Pi$ such that $\sigma a$ divides $b$.  Equivalently $a |_\Pi b$ if $b \in \ideal{a}_\Pi$.
\end{definition}

Recall that a monomial order on $R = K\mon$ is a total order $\leq$ on $\mon$ that is a well-order, and that respects multiplication, meaning if $a \leq b$ then $ac \leq bc$ for all $c \in \mon$.

\begin{definition}
 A monomial order $\leq$ on $R = K\mon$ {\em respects $\Pi$} if whenever $a \leq b$ then $\sigma a \leq \sigma b$ for all $\sigma \in \Pi$.
\end{definition}

Therefore order $\leq$ is a $\Pi$-respecting monomial order on $R$ if $\leq$ is a total well-order on $M$ that respects the action of $\mon*\Pi$.
Now we have the tools to describe the $\Pi$-equivariant version of Gr\"obner bases.
\begin{definition}
 Let $R = K\mon$ be a monoid ring with $\Pi$ action on $\mon$, and let $\leq$ be a $\Pi$ respecting monomial order.  Given a $\Pi$-invariant ideal $I \subseteq R$, a {\em $\Pi$-equivariant Gr\"obner basis} of $I$ is a set $G \subseteq I$ such that the $\Pi$ orbits of $G$ form a Gr\"obner basis of $I$,
 \[ \ideal{\LT \Pi G} = \LT I. \]
\end{definition}
We require $\leq$ to be a $\Pi$ respecting order because it is equivalent to the condition that
\[ \LT \sigma f = \sigma \LT f \]
for all $f \in R$ and $\sigma \in \Pi$.  Therefore with such an order
 \[ \ideal{\LT G}_{\Pi} = \ideal{\LT \Pi G} = \LT I. \]
This also implies that $\LT I$ is a $\Pi$-invariant ideal.  Note that since the $\Pi$ orbits of $G$ are a Gr\"obner basis of $I$, we have $\ideal{G}_\Pi = I$.

\begin{proposition}[Remark 2.1 of \cite{Brouwer09e}]\label{prop:nogroup}
 Let $\Pi$ be a group which acts non-trivially on $\mon$.  Then $K\mon$ has no $\Pi$ respecting monomial orders.
\end{proposition}
\begin{proof}
 Suppose that $\leq$ is a $\Pi$ respecting order and choose $\sigma \in \Pi$ and $m \in \mon$ such that $m \neq \sigma m$.  If $m > \sigma m$ then $\sigma^n m > \sigma^{n+1} m$ for all $n$ so
  \[ m > \sigma m > \sigma^2 m > \cdots \]
 is an infinite descending chain of monomials, contradicting the fact that $\leq$ is a well-order.  If $m < \sigma m$ then $m > \sigma^{-1} m > \sigma^{-2} m > \cdots$ is an infinite descending chain.
\end{proof}

In particular this means that $R$ with non-trivial $\Sym$ action has no $\Sym$ respecting monomial orders.  To deal with this problem, a related monoid is introduced to replace $\Sym$ which allows for monomial orders but is somehow large enough compared to $\Sym$ not to break properties like finite generation.

Define the {\em monoid of strictly increasing functions} as
\[ \Inc := \{ \rho:\B N \to \B N \mid \text{ for all } a < b, \rho(a) < \rho(b) \}. \]
For any $\Sym$-algebra $R$, there is a natural action of $\Inc$ on $R$ as follows.  Fixing $f \in R$, for any $\sigma \in \Sym$ the value of $\sigma f$ depends only on the restriction $\sigma|_{[w(f)]}$ considering $\sigma$ as a function $\B N \to \B N$.  For any $\rho \in \Inc$ there exists $\sigma \in \Sym$ such that $\sigma|_{[w(f)]} = \rho|_{[w(f)]}$ and define $\rho f = \sigma f$.  It can be checked that this gives a well-defined action of $\Inc$ on $R$.

It immediately follows from the definition of the action that $\Inc f \subseteq \Sym f$.  Despite the fact that $\Inc$ is not a submonoid of $\Sym$, it behaves like one in terms of its action on $R$.  An injective map $\sigma|_{[w(f)]}: [w(f)] \to \B N$ can always be factored into $\rho' \circ \tau$ with $\tau \in \F S_{w(f)}$ and $\rho':[w(f)] \to \B N$ a strictly increasing function.  The map $\rho'$ can be extended to some $\rho \in \Inc$, and then $\sigma f = \rho (\tau f)$.  Thus
 \[ \Sym f = \bigcup_{\tau \in \F S_{w(f)}} \Inc(\tau f). \]
This fact that the $\Sym$-orbit of any $f$ is a finite union of $\Inc$-orbits implies the following statements.

\begin{proposition}
 Let $R$ be a $\Sym$-algebra with the finite width property and let $I\subseteq R$ be a $\Sym$-invariant ideal.
 \begin{itemize}
  \item $I$ is $\Inc$-invariant.
  \item $I$ is $\Sym$-finitely generated if and only if $I$ is $\Inc$-finitely generated.
  \item If $R$ is $\Inc$-Noetherian then $R$ is $\Sym$-Noetherian.
 \end{itemize}
\end{proposition}

When computing Gr\"obner bases of $\Sym$-invariant ideals we will work with the $\Inc$ action instead.  If $G$ is an $\Inc$-equivariant Gr\"obner basis for $\Sym$-invariant ideal $I$, then the $\Sym$-orbits of $G$ also form a Gr\"obner basis of $I$.  Generally the rings we are interested in will have $\Inc$ respecting monomial orders.

\begin{example}
 Let $R = K[x_1,x_2,\ldots]$ with $\Inc$-action defined by $\rho \cdot x_i = x_{\rho(i)}$.  The lexicographic order $\leq$ on the monomials of $R$ with $x_1 < x_2 < x_3 < \cdots$ is a $\Inc$ respecting monomial order.  This is the only possible lexicographic order on $R$ that respects $\Inc$.  There are also a graded lexicographic and a graded reverse lexicographic order on $R$ that respect $\Inc$.  There is no $\Inc$ respecting monomial order on $R$ that is defined by a single weight vector in $\B R^{\B N}$.
\end{example}

It is an open question to characterize all possible $\Inc$ respecting monomial orders on a given ring $K\mon$ with $\Inc$ action.  We can make the following statement about such orders.

\begin{proposition}
 If $\leq$ is a $\Pi$ respecting monomial order on $K\mon$, then $\leq$ refines the $\Pi$-divisibility quasi-order $|_\Pi$.
\end{proposition}
\begin{proof}
 Suppose $a$ and $b$ are monomials with $a |_\Pi b$, so there is some pair $\sigma \in \Pi$, $c \in \mon$ such that $c\sigma a = b$.  From the proof of Proposition \ref{prop:nogroup} we see that $a \leq \sigma a$.  Since $1 \leq c$ and $\leq$ respects multiplication, $\sigma a \leq c\sigma a = b$.
\end{proof}

An implication of this proposition is that if $K\mon$ has a $\Pi$ respecting monomial order then the $\Pi$-divisibility quasi-order must be a partial order, meaning it has the anti-symmetry property: if $a |_\Pi b$ and $b |_\Pi a$ then $a = b$.

If $R$ is $\Pi$-Noetherian with a $\Pi$ respecting monomial order, then any $\Pi$-invariant ideal $I \subseteq R$ will have a finite $\Pi$-equivariant Gr\"obner basis.  This follows from the fact that $\LT I$ is $\Pi$-finitely generated.  We recount two prior results that give examples of $\Inc$-Noetherian rings, and will be directly relevant to the remainder of this paper.

\begin{theorem}[Theorem 1.1 of \cite{hillar2012finite}]\label{thm:HS}
 Let $X = \{x_{ij} \mid i \in [k], j \in \B N\}$, and let $\Sym$ act on $[X]$ by permuting the second index, $\sigma x_{ij} = x_{i\sigma(j)}$ for $\sigma \in \Sym$.  Then $K[X]$ is $\Inc$-Noetherian.
\end{theorem}

\begin{theorem}[Theorem 1.1 of \cite{draisma2013noetherianity}]\label{thm:DEKL}
 Let $K[Y]$ be a $\Sym$-finitely generated $\Sym$-algebra with the finite width property, and with $\Sym$ action on variable set $Y$.  For $K[X]$ defined as in Theorem \ref{thm:HS}, let $\phi$ be a $\Sym$-equivariant monomial map
  \[ \phi: K[Y] \to K[X]. \]
 Then
 \begin{itemize}
  \item $\ker \phi$ is $\Inc$-finitely generated,
  \item $\im \phi$ is $\Inc$-Noetherian.
 \end{itemize}
\end{theorem}

The conditions on ring $K[Y]$ in Theorem \ref{thm:DEKL} are quite general and Hillar and Sullivant \cite{hillar2012finite} prove that such rings are generally not $\Sym$-Noetherian.  They show that Noetherianity fails for the example $K[Y]$ where $Y = \{y_{ij} \mid i, j \in \B N\}$ with $\sigma y_{ij} = y_{\sigma(i)\sigma(j)}$ for $\sigma \in \Sym$.

When $R$ is not $\Pi$-Noetherian, we do not know in general if a $\Pi$-finitely generated ideal $I\subseteq R$ has a finite $\Pi$-equivariant Gr\"obner basis, or if so, for which monomial orders.  However we will prove in Section \ref{sec:egbtoric} that the $\Sym$-invariant toric ideal $\ker \phi$ as in Theorem \ref{thm:DEKL} does have finite $\Inc$-equivariant Gr\"obner bases for specifically chosen monomial orders.  This will allow an algorithm to compute such a Gr\"obner basis of $\ker \phi$, given the map $\phi$.

\section{Equivariant Buchberger algorithm}
\subsection{Description of the algorithm}
An equivariant adaptation of Buchberger's algorithm was first proposed in \cite{aschenbrenner2007finite} and formalized in \cite{Brouwer09e} and \cite{draisma2013noetherianity}.

Let $R = K\mon$ with $\Pi$ acting on $\mon$, and let $\leq$ be a $\Pi$ respecting monomial order.
For $G \subseteq R$, a {\em $\Pi$-normal form} of $f$ with respect to $G$ denoted $\NF_{\Pi G}(f)$ is the result of repeated reductions of $f$ by elements of $\Pi G$ until no more reductions are possible.

The equivariant Buchberger algorithm departs from the conventional Buchberger algorithm only at the step of adding new S-pairs to the list $S$.  This departure is described after, along with a definition of $O_{f,g}$ and the ``finite S-pair condition.''

\begin{algorithm}[Brouwer--Draisma \cite{Brouwer09e}]\label{alg:Buchberger}
$G = \alg{Buchberger}(F)$
\begin{algorithmic}[1]
\REQUIRE $F$ is a finite set of elements in $R = K\mon$ with $\Pi$ acting on $\mon$ and satisfying the finite S-pair condition.
\ENSURE $G$ is $\Pi$-equivariant Gr\"obner basis of $\ideal{F}_{\Pi}$.

\smallskip \hrule \smallskip

\STATE $G\gets F$
\STATE $S\gets \bigcup_{f,g\in G} O_{f,g}$
\WHILE{$S\neq\emptyset$}
	\STATE pick $(h_1,h_2) \in S$
	\STATE $S\gets S\setminus\{(h_1,h_2)\}$ 
	\STATE $h \gets \NF_{\Pi G}(h_1 - \frac{\LC(h_1)}{\LC(h_2)}h_2)$
  	\IF{$h \neq 0$}
		\STATE $G\gets G\cup \{h\}$
		\STATE $S\gets S\cup \left(\bigcup_{g\in G}O_{g,h}\right)$
	\ENDIF
\ENDWHILE
\smallskip \hrule \smallskip
\end{algorithmic}
\end{algorithm}

Given $f,g \in R$ define
 \[ \C S_{f,g} := \{(m_1f,m_2g) \mid m_1,m_2 \in \mon * \Pi \text{ such that } \LT m_1f = \LT m_2g\}. \]
This set is closed under the diagonal action of $\mon *\Pi$ making $\C S_{f,g}$ a $\mon *\Pi$-set.  A set $G \subseteq R$ satisfies the {\em equivariant Buchberger criterion} if for all $(h_1,h_2) \in \bigcup_{f,g\in G} \C S_{f,g}$,
 \[ \NF_{\Pi G}(\LC(h_2)h_1 - \LC(h_1)h_2) = 0. \]
The set $G$ is a $\Pi$-equivariant Gr\"obner basis of $\ideal{G}_{\Pi}$ if and only if it satisfies the equivariant Buchberger criterion.  The proof of this fact follows by applying the usual Buchberger criterion to the set $\Pi G$ (see Theorem 2.5 of \cite{Brouwer09e}).

For each pair $f,g \in G$, it is sufficient to check the equivariant Buchberger criterion on a $\mon * \Pi$ generating set of $\C S_{f,g}$, which we denote $O_{f,g}$.

\begin{definition}
 A $\Pi$-algebra $R = K\mon$ has the {\em finite S-pair condition} if for any $f,g \in R$, the set $\C S_{f,g}$ is finitely generated as a $\mon * \Pi$-set.  In \cite{Brouwer09e} this condition is referred to as ``EGB4.''
\end{definition}

If $R$ fails the finite S-pair condition then we cannot apply the equivariant Buchberger algorithm.
When $\Pi$ is trivial and $R$ is a polynomial ring (the setting of the conventional Buchberger algorithm), $\C S_{f,g}$ is generated by a single pair
 \[ (\frac{\lcm(\LT f, \LT g)}{\LT(f)} f, \frac{\lcm(\LT f, \LT g)}{\LT(g)} g).\]
This generator is called the ``S-pair'' of $f,g$.  Therefore $R$ in this case satisfies the finite S-pair condition, and the equivariant Buchberger algorithm specializes to the conventional Buchberger algorithm.

\begin{proposition}
 If $\Sym$-algebra $R$ is a polynomial ring with the finite width condition, then $R$ has the finite S-pair condition.
\end{proposition}
\begin{proof}
 Fix $f,g \in R$.
 Since $R$ is a polynomial ring, for fixed $\sigma_1,\sigma_2 \in \Inc$, all elements of $\C S_{f,g}$ of the form $(m_1\sigma_1 f, m_2\sigma_2 g)$ with $m_1,m_2 \in \mon$ are monomial multiplies of the usual S-pair of $\sigma_1 f, \sigma_2 g$,
  \[ \left(\frac{m}{\LT \sigma_1 f} \sigma_1 f,\; \frac{m}{\LT \sigma_2 g} \sigma_2 g\right) \]
 where $m = \lcm(\LT \sigma_1 f,\LT \sigma_2 g)$.

 Any $f,g \in R$ have finite width so $\sigma_1 f$ depends only on $\sigma_1|_{[w(f)]}$, and similarly for $\sigma_2 g$.  In fact we can always factor the pair as
  \[ (\sigma_1 f, \sigma_2 g) = \tau(\rho_1 f, \rho_2 g) \]
 for some $\tau \in \Inc$, while $\rho_1:[w(f)] \to [w(f) + w(g)]$ and $\rho_2:[w(g)] \to [w(f) + w(g)]$ are strictly increasing functions.  Here $\rho_1$ and $\rho_2$ are chosen to ``interlace'' the variables of $f$ and $g$ in the same way as $\sigma_1,\sigma_2$.  (To consider $\rho_1,\rho_2$ as elements of $\Inc$, take any choice of extensions to maps on $\B N$.)
 
 Then $\C S_{f,g}$ is generated by the finite set of pairs of the form
  \[ \left(\frac{m}{\LT \rho_1 f} \rho_1 f,\; \frac{m}{\LT \rho_2 g} \rho_2 g\right) \]
 with $\rho_1:[w(f)] \to [w(f) + w(g)]$ and $\rho_2:[w(g)] \to [w(f) + w(g)]$ where $m = \lcm(\LT \rho_1 f,\LT \rho_2 g)$.
\end{proof}

We note that Algorithm \ref{alg:Buchberger} is guaranteed to terminate when $R$ is $\Pi$-Noetherian.  Without Noetherianity, a finite Gr\"obner basis might not exist, and when it does it is not known in general if the equivariant Buchberger algorithm will terminate.

\subsection{A terminating $\Inc$-equivariant Gr\"obner basis algorithm}
The algorithm concept in this section has been used implicitly before in computations by Jan Draisma, Anton Leykin and perhaps others, but does not appear in the literature (at least to the knowledge of the author) so it is given a presentation here.

Let $R = K\mon$ with $\Inc$ action on $\mon$, $R$ satisfying the finite S-pair condition, and with each truncation $R_n$ a Noetherian ring.  (These conditions are satisfied for example when $R$ is $\Inc$-finitely generated polynomial ring, as in Theorem \ref{thm:DEKL}.)   Let $I \subseteq R$ be a $\Inc$-invariant ideal which is $\Inc$-generated by finite set $F$, and moreover has finite $\Inc$-equivariant Gr\"obner basis $G$.  Define the {\em generator truncation} of $I$ to be $\tilde{I}_{F,n} := \ideal{\Inc F \cap R_n} \cap R_n$.  Note that $\tilde{I}_{F,n} \subseteq I_n$ but in general equality does not hold.  For $f \in I$ define $w_F(f)$ to be the minimum value of $n$ for which $f \in \tilde{I}_{F,n}$.

The truncated EGB algorithm again takes finite generating set $F$ as its input.  For each successive $n \geq w(F)$, compute a set $G_n$ such that $\Inc G_n \cap R_n$ is a Gr\"obner basis for $\tilde{I}_{F,n}$.  Then check if $G_n$ is a $\Inc$-equivariant Gr\"obner basis of $I$ using the equivariant Buchberger criterion, and if so return $G_n$.

\begin{algorithm}\label{alg:truncBuch}
$G = \alg{TruncatedEGB}(F)$
\begin{algorithmic}[1]
\REQUIRE $F$ is a finite set of elements in $R = K\mon$ with $\Inc$ acting on $\mon$, $R$ satisfies the finite S-pair conditions, and each $R_n$ is Noetherian.
\ENSURE $G$ is a $\Inc$-equivariant Gr\"obner basis of $I := \ideal{F}_{\Inc}$.

\smallskip \hrule \smallskip

\STATE $G\gets F$
\STATE $n\gets w(F)$
\WHILE{$G$ not a $\Inc$-equivariant Gr\"obner basis of $I$}
	\STATE $G\gets$ Gr\"obner basis of $\tilde{I}_{F,n}$
	\STATE $n \gets n+1$
\ENDWHILE
\smallskip \hrule \smallskip
\end{algorithmic}
\end{algorithm}

\begin{proof}[proof of termination]
For each $n$, let $G_n$ denote the value of $G$ after that step.  Computing $G_n$ is a finite process since it takes place in $R_n$ which is Noetherian.  $G_n$ is a finite set and so it has a finite number of S-pairs to be checked.  Therefore testing whether $G_n$ is a $\Inc$-equivariant Gr\"obner basis is finite.

It remains to be proved that $G_n$ is a $\Inc$-equivariant Gr\"obner basis for some value of $n$.  If $H$ is a $\Inc$-equivariant Gr\"obner basis of $I$, for any $h \in H$ we have $h \in \tilde{I}_{F,n}$ for all $n \geq w_F(h)$, so $\LT(h) \in \LT(\tilde{I}_{F,n})$.  Therefore there is some element $g \in G_n$ with $\LT(g)|_{\Inc} \LT(h)$.  For $n = \max_{h\in H} w_F(h)$, the initial ideal $\ideal{\LT(G_n)}_{\Inc}$ contains $\ideal{\LT(H)}_{\Inc}$ and so $G_n$ is a $\Inc$-equivariant Gr\"obner basis of $I$.
\end{proof}

In practice, $G_n$ can be computed either using a traditional Gr\"obner basis algorithm on input $\Inc F \cap R_n$, or using an equivariant Buchberger algorithm on input $F$ with the following two caveats:
\begin{itemize}
 \item consider only S-pairs $(m_1f, m_2g)$ with $m_1f$ and $m_2g$ both having width $\leq n$,
 \item perform only reductions such that the outcome has width $\leq n$.
\end{itemize}
Moreover we do not need to restart the algorithm from scratch for each $n$.  Instead $G_{n-1} \cup F$ can be used as the input for the $n$th step instead of $F$.

If $\leq$ is a width order (a monomial order such that $w(a) < w(b)$ implies $a < b$), the second condition is satisfied automatically since reductions cannot increase the width.  Therefore the normal form of a given S-pair does not depend on $n$ and only needs to be computed once.  As a result we can use Algorithm \ref{alg:Buchberger}, queuing S-pairs by width so that the smallest width S-pairs are considered first.  The algorithm terminates once the queue is empty.  A separate check for whether $G_n$ is a $\Inc$-equivariant Gr\"obner basis for $I$ is not needed since this is equivalent to reducing all S-pairs in the queue.

\section{Symmetric Gr\"obner bases of toric ideals}\label{sec:egbtoric}
The previous section an algorithm was given that is capable of computing a n $\Inc$-equivariant Gr\"obner basis of an ideal in a ring of the form $K[Y]$ with $\Inc$ action on $Y$ and $Y$ having a finite number of orbits, with guaranteed termination {\em if} a finite Gr\"obner basis for the ideal exists.  In this section we prove that any $\Sym$-invariant toric ideal $\ker \phi$ of form in Theorem \ref{thm:DEKL} has a finite Gr\"obner basis with respect to a particularly chosen monomial order.  We then show that a Gr\"obner basis can be computed given the monomial map $\phi$, using elimination.

\subsection{Existence of equivariant Gr\"obner bases of toric ideals}
As in Theorem \ref{thm:DEKL}, let $K[Y]$ be a $\Sym$-finitely generated polynomial ring with $\Sym$ action on the variable set $Y$.  Let $X = \{x_{i,j} \mid i \in [\ell], j \in \B N\}$ with $\Sym$ acting on the second index of each $x_{i,j}$.  Let $\phi:K[Y] \to K[X]$ be a $\Sym$-equivariant monomial map.

\begin{theorem}\label{theo:finGB}
 $\ker \phi$ has a finite $\Inc$-equivariant Gr\"obner basis $H$ with respect to a $\Inc$-respecting monomial order $\leq$ that can be constructed based on $\phi$.
\end{theorem}

A key strategy to studying $\phi$ used in \cite{draisma2013noetherianity} was working with a certain factorization of the map, and we make use of this strategy here as well.

Let $N$ be the number of $\Sym$-orbits in $Y$, and for the $p$th orbit choose a representative $y_p$ with minimum width and let $k_p = w(y_p)$.  Define a new variable set
 \begin{equation}\label{yprime}
 Y' := \{y'_{p, (\alpha_1,\ldots,\alpha_{k_p})} \mid p \in [N],\; \alpha_1,\ldots,\alpha_{k_p} \in \B N \text{ distinct } \}
 \end{equation}
with $\Sym$ action on $Y'$ by
 \[ \sigma y'_{p, (\alpha_1,\ldots,\alpha_{k_p})} = y'_{p, (\sigma(\alpha_1),\ldots,\sigma(\alpha_{k_p}))}.\]
$Y'$ has the same number of $\Sym$-orbits as $Y$ and minimal-width orbit representatives with the same widths, but $Y'$ has a standard form we will take advantage of.  There is a surjective $\Sym$-equivariant map $\theta:Y' \to Y$ defined by $y'_{p, (1,\ldots,k_p)} \mapsto y_p$, and this map extends to monomial map $\theta: K[Y'] \to K[Y]$.

Define variable set
 \[ Z := \{z_{p,i,j} \mid p \in [N], i \in [k_p], j \in \B N\}\]
with $\Sym$ action on the last index $\sigma z_{p,i,j} = z_{p,i,\sigma(j)}$.  Define $\Sym$-equivariant monomial map $\pi:K[Y'] \to K[Z]$ by
 \[ \pi: y'_{p, (\alpha_1,\ldots,\alpha_{k_p})} \mapsto \prod_{i \in [k_p]} z_{p,i,\alpha_{i}}. \]
Note that this map $\pi$ depends only on the values of $k_1,\ldots,k_p$.  The important fact is that $\phi \circ \theta$ factors through the map $\pi$, so there is a $\Sym$-equivariant monomial map $\psi:K[Z] \to K[X]$ that makes the following diagram commute.
\[
\xymatrix{
K[Y'] \ar[r]^{\pi} \ar[d]^{\theta} & K[Z] \ar[d]^{\psi} \\
K[Y] \ar[r]^{\phi} & K[X]
}
\]

\begin{proposition}\label{prop:theta}
 If $G \subseteq K[Y']$ is a $\Inc$-equivariant Gr\"obner basis of $\ker(\phi \circ \theta)$ for $\Inc$-respecting monomial order $\leq'$, then $\theta(G) \subseteq K[Y]$ is a $\Inc$-equivariant Gr\"obner basis of $\ker \phi$ for a $\Inc$-respecting monomial order $\leq$ determined by $\leq'$.
\end{proposition}
\begin{proof}
 Let $\nu:K[Y] \to K[Y']$ be the right-inverse map of $\theta$ defined on variables $y \in Y$ by letting $\nu(y)$ be the $\leq'$-minimal variable in $\theta^{-1}(y)$.  Define order $\leq$ on $[Y]$ by letting $a \leq b$ if and only if $\nu(a) \leq' \nu(b)$.  It is clear that $\leq$ is a total well-order on $[Y]$, and it respects multiplication because $\nu$ is a homomorphism.  Fixing any $\rho \in \Inc$ and $a \in [Y]$ note that $\rho \nu(a)$ and $\nu(\rho a)$ are in the same fiber of $\theta$.  There exists $\sigma \in \Sym$ such that $\sigma \rho$ is the identity on $[w(a)]$ and so $\sigma \nu(\rho a) \in \theta^{-1}(a)$.  This implies $\nu(a) \leq' \sigma \nu(\rho a)$ and $\nu(\rho a) = \rho \sigma \nu(\rho a) \leq' \rho \nu(a)$.  Because $\leq'$ respects $\Inc$ it must be that $\nu(\rho a) = \rho \nu(a)$.  Therefore $\leq$ respects $\Inc$.
 
 Let $f$ be any non-zero polynomial in $\ker \phi$.  Then $\LT f = \theta \LTp \nu(f)$  This polynomial $\nu(f)$ is in $\ker(\phi \circ \theta)$ so there is $g \in \Inc G$ that reduces $\nu(f)$.  Since $\LTp g$ divides $\LTp \nu(f)$ which is minimal in its fiber, $\LTp g$ is a product of minimal variables so it is also minimal in its fiber.  Because $\LTp g$ is larger than the other monomials in $g$, it is certainly larger than the minimal representatives of the fibers of those monomials.  Therefore $\LT(\theta(g)) = \theta(\LTp g)$ so $\theta(g) \in \Inc \theta(G)$ reduces $f$.
\end{proof}

The above proposition implies that to prove Theorem \ref{theo:finGB} it is sufficient to prove it for maps of the form $\psi \circ \pi:K[Y'] \to K[X]$.  Therefore we can assume without loss of generality that $Y = Y'$ and $\phi$ factors as
\[ R[Y] \xrightarrow{\pi} R[Z] \xrightarrow{\psi} R[X]. \]

We now construct the monomial order $\leq$ on $K[Y]$ for which we will prove the existence of a finite equivariant Gr\"obner basis.
Let $\mm$ denote the monoid of monomials $\pi[Y] \subseteq [Z]$, called a {\em matching monoid}.  Choose an $\Inc$ respecting monomial order $\leq_1$ on $K\mm$ and an $\Inc$ respecting reverse lexicographic order $\leq_2$ on $K[Y]$.  Let $\leq$ be the monomial order on $K[Y]$ defined by $a < b$ if $\pi(a) <_1 \pi(b)$ or $\pi(a) = \pi(b)$ and $a <_2 b$.

We will first prove the existence of a finite $\Inc$-equivariant Gr\"obner basis of $\ker \pi$ for order $\leq$.  The following proposition will allow us to do so by bounding the degree of a Gr\"obner basis.

\begin{proposition}\label{prop:bounded-deg}
 Let $F \subset K[Y]$ be a set of binomials with degree bounded by $d$.  Then $F$ is contained in a finite number of $\Inc$-orbits.
\end{proposition}
\begin{proof}
 Let $k$ be the maximum size of the index supports of variables in $Y$.  The index support of a polynomial is the union of the index supports of its variables, so a binomial of degree $\leq d$ has index support size bounded by $2kd$.  Any $f \in F$ has in its $\Sym$-orbit a binomial of width $\leq 2kd$.  There are only a finite number of pairs of monomials in $K[Y]_{2kd}$ with degree $\leq d$.  Therefore $F$ is contained in a finite number of $\Sym$-orbits, and every $\Sym$-orbit is the union of a finite number of $\Inc$-orbits.
\end{proof}

We also require the following useful characterization of the matching monoid $\mm$.
Any monomial in $[Z]$ can be expressed as $z^A$ with $A := (A_1,\ldots,A_p)$ where each $A_p$ is a $k_p$ by infinite matrix.

\begin{lemma}[Proposition 3.1 of \cite{draisma2013noetherianity}]\label{lem:mm}
 Monomial $z^A \in [Z]$ is in $\mm$ if and only if for each $p = 1,\ldots,N$ there is integer $d_p \geq 0$ such that all row sums of $A_p$ are equal to $d_p$ and all column sums are $\leq d_p$.
\end{lemma}
A consequence is that for $z^A,z^B \in \mm$, $z^A$ divides $z^B$ in $\mm$ (meaning there is a quotient $z^C \in \mm$) if and only if $A \leq B$ entry-wise and for each $p = 1,\ldots,N$ every column sum of $B_p-A_p$ is bounded by $d'_p - d_p$ where $d'_p$ and $d_p$ are the row sums of $B_p$ and $A_p$ respectively.

\begin{proposition}\label{prop:gb-of-phi}
The kernel of $\pi$ has a Gr\"obner basis for order $\leq$ consisting of binomials of degree at most $2 \max_p k_p - 1$.
\end{proposition}

The argument of the following proof is due to Jan Draisma, originally used to show a degree bound on generators of $\ker \pi$.  We adapt it to Gr\"obner bases.

\begin{proof}
Given monomial $v \in [Y]$, let $u$ be the minimal monomial with $\pi(u) = \pi(v)$, i.e. the standard monomial of $\ker \pi$ in the equivalence class of $v$.  It suffices to show that there exists a chain $v=v_0 > v_1 > \ldots > v_t=u$ such that $\pi(v_s)=\pi(u)$ for all $s$ and $(v_s - v_{s+1})/\gcd(v_s,v_{s+1})$ has degree at most $2 \max_p k_p - 1$.  We first consider the case $N=1$ and drop the index $p$ for simplicity of notation.

Proceed by induction on the degree of $v$.  Suppose $v$ and $u$ have a variable $y_{J}$ in common.  By the induction hypothesis there is a chain from $v/y_{J}$ to $u/y_{J}$ satisfying the desired conditions, since $u/y_{J}$ is also a standard monomial.  Multiplying by $y_J$ gives a chain from $v$ to $u$.

Assume then that $v$ and $u$ have no variables in common.  Let $z^A = \pi(v) = \pi(u)$.  Let $y_{J}$ be the smallest variable in $u$ and $z^B = \pi(y_{J})$.  The matrix $B$ has $k$ non-zero entries and row sums equal to 1.  Since $A \geq B$, monomial $v$ has a divisor $v'$ of degree $e \leq k$ such that $\pi(v')=:z^{A'}$ and $A' \geq B$.  The row sums of $A'$ are equal to $e$ and let $S$ be the set of indices for which the column sum of $A'$ is also equal to $e$.  Abusing notation let $J$ also denote the set of column indices where $B$ is non-zero.  By Lemma \ref{lem:mm}, $z^B$ divides $z^{A'}$ in $\mm$ if and only if $S$ is contained in $J$, because then all column sums of $A'-B$ will be $\leq e-1$.  The column indices $S\setminus J$ are the obstacles to divisibility, but there are few of them.  The total entry sum of $A'$ is $ek$, and the columns of $J$ make a non-trivial contribution to the sum, so then $|S\setminus J| < k$.

For each $j \in S \setminus J$, the $j$th column sum of $A$ is strictly less than $\deg u$ because $y_J$ does not contribute.  Therefore $v$ also has a variable $y_{L_j}$ with $j$ not among the entries of $L_j$.  Let $w$ be the product of the set of variables of the form $y_{L_j}$ for some $j \in S \setminus J$ and let $v'' = v'w$.  Letting $z^{A''} = \pi(v'')$, the set of indices for which the column sum of $A''$ is equal to $\deg v''$ is contained in $J$. Therefore $A'' - B$ has all column sums $\leq \deg v'' - 1$ and so $z^{A''-B} \in \mm$.  Choose any monomial $v_1'' \in \pi^{-1}(z^{A'-B})$, and let $v_1 = v\frac{v''_1 y_J}{v''}$.  By construction, $\pi(v_1)=\pi(v)$ and $v_1$ shares the variable $y_{J}$ with $u$.  Note that since $u < v$ and $\leq$ is reverse lexicographic on each fiber of $\pi$, every variable in $v$ is larger than $y_{J}$, so $v > v_1$ as well.  Since $v_1$ and $u$ share a variable, by the induction hypothesis there is a chain from $v_1$ to $u$ satisfying the desired conditions.  The degree of $(v-v_1)/\gcd(v,v_1)$ is bounded by $\deg v''$ which is $\leq 2k-1$.

For the case $N > 1$ group $v$ and $u$ by their variable orbits as $v = v_1\cdots v_N$ and $u = u_1\cdots u_N$.  By the above, each $v_p$ can be reduced to $u_p$ with reductions by Gr\"obner basis elements of degree $\leq 2 \max_p k_p - 1$.  Applying these reductions in sequence brings $v$ to $u$.
\end{proof}

Let $F$ denote such a finite $\Inc$-equivariant Gr\"obner basis of $\ker \pi$.  We also know there exists $G$, a finite $\Inc$-equivariant Gr\"obner basis of $\ker \psi \cap \im \pi$ with respect to $\leq_1$, because $\im \pi$ is $\Inc$-Noetherian by Theorem \ref{thm:DEKL}.  The goal is to combine $F$ with a ``lift'' of $G$ to form a Gr\"obner basis of $\ker \phi$, and then show that this Gr\"obner basis has bounded degree.

\begin{lemma}\label{lem:lift}
 Let $z^A,z^B$ be any pair of monomials in $\mm$, and $u$ any monomial in $\pi^{-1}(z^A)$.  There exists monomial $v \in \pi^{-1}(z^B)$ such that
  \[ \degtot \frac{u - v}{\gcd(u,v)} \leq 5\norm{A-B}. \]
\end{lemma}
\begin{proof}
  First we consider the case $N=1$ and drop the index $p$ for simplicity of notation.  We will prove that $u$ has a large factor $u'$ such that $\pi(u') \mid z^B$.  Then $v$ can be chosen to be $v = u'v''$ where $v''$ is any monomial in $\pi^{-1}(z^B/\pi(u'))$.  The sizes of $v'' = v/u'$ and $u'' = u/u'$ will be small.
  
  Let $n = \deg u$ and $m = \deg \pi^{-1}(z^B)$.  Express $u$ as the product of a sequence of variables $u = \prod_{i=1}^{n} y_{J_i}$, and let $u_j$ denote the partial product $u_j = \prod_{i=1}^{j} y_{J_i}$.  Let $A_j$ be the exponent matrix of $\pi(u_j)$ noting that $A_n = A$.  For any exponent matrix $C$, $C_{+\ell}$ will denote the $\ell$th column sum of $C$.  We will throw out variables $y_{J_i}$ which cause obstructions to $\pi(u)$ dividing $z^B$.  Let $L$ be the set of all $j \in [n]$ such that none of the following are true.
   \begin{enumerate}
    \item[(1)] $j > m$,
    \item[(2)] the index sequence $J_j = (\alpha_1,\ldots,\alpha_k)$ has an element $\alpha_i$ such that $(A_j)_{i,\alpha_i} > B_{i,\alpha_i}$,
    \item[(3)] there exists $\ell$ not in the index sequence $J_j$ such that $j - (A_j)_{+\ell} > m - B_{+\ell}$.
   \end{enumerate}
  As $j$ increases, $(A_j)_{i,\ell}$ increases exactly for $y_{J_j}$ such that $\ell$ is the $i$th index of $J_j$, and $j - (A_j)_{+\ell}$ increases exactly when $\ell$ does not appear in $J_j$.  Letting $u' := \prod_{i \in L}  y_{J_i}$ and $A'$ be the exponent matrix of $\pi(u')$, we have $A' \leq B$ and $\deg(u') - A'_{+\ell}\leq m - B_{+\ell}$ for all $\ell$.   By Lemma \ref{lem:mm} $\pi(u')$ divides $z^B$ in $\mm$.
  We will bound $|L| = \deg(u')$ by bounding the number of values $j \in [n]$ that satisfy each case above.
  
  {\em Case (1):} The number of values of $j \in [n]$ with $j > m$ is bounded by $|n-m|$.
  
  {\em Case (2):} The number of values of $j \in [n]$ satisfying condition (2) is bounded by
   \[ \sum_{i,\ell} |A_{i,\ell} - B_{i,\ell}| = \norm{A - B}. \]
  
  {\em Case (3):} Here we consider $j$ satisfying condition (3) only for $j \leq m$.  Let $n' = \min(m,n)$.  The number of such values of $j$ is bounded by
   \[ \sum_\ell \max\{0, B_{+\ell} - (A_{n'})_{+\ell}\} + n' - m \leq \norm{A_{n'} - B} \]
   \[ \leq \norm{A-B} + \norm{A_{n'} - A} \leq \norm{A-B} + k|n-m|.\]
  The sum of all entries in $A$ is $kn$ and for $B$ is $km$ which shows that $k|n-m| \leq \norm{A-B}$.  Therefore \[\deg (u'') = n-|L| \leq (3+1/k)\norm{A-B},\] \[\deg(v'') = m-|L| \leq n-|L| - |n-m| \leq (3+2/k)\norm{A-B}.\] So then
   \[ \deg \frac{u - v}{\gcd(u,v)} \leq \deg(u'' - v'') \leq (3+2/k)\norm{A-B} \leq 5\norm{A-B}. \]
  
  For $N >1$, for each $p \in [N]$ construct $v_p \in \pi^{-1}(y^{B_p})$ such that $\deg_p \frac{u_p - v_p}{\gcd(u_p,v_p)} \leq 5\norm{A_p-B_p}$.  Then take $v = v_1\cdots v_N$.
\end{proof}

With some additional care, the factor 5 in the bound can be significantly improved but a more precise bounded isn't needed here.

Let $C$ be the maximum value of $\norm{A-B}$ over all binomials $z^A - z^B$ in $G$.  Let $\C G$ be the set of all binomials $f$ in $K[Y]$ such that $\pi(f) = \sigma wg$ for some $\sigma \in \Inc$, $w \in \mm$ and $g \in G$, and such that $\degtot f \leq 5C$.  In other words $\C G$ consists of all lifts of elements of $\Inc G$ of bounded degree.  Finally let $\C H = \Inc F \cup \C G$.
\begin{proposition}
 $\C H$ is a Gr\"obner basis of $\ker \phi$ with respect to $\leq$.
\end{proposition}
\begin{proof}
 It's clear that $\C H \subseteq \ker \phi$.  We will show that any $u \in [Y]$ which is not a standard monomial of $\ker \pi$ can be reduced by $\C H$.
 
 First consider the case that $\pi u$ is not a standard monomial of $\ker \psi$.  There is a binomial $g \in G$ which reduces $\pi u$, which we will write as $g = a-b$ with $a$ the lead term.  There is $\sigma \in \Inc$, $c \in \mm$ such that $\sigma c a = \pi u$.  By Lemma \ref{lem:lift} there is $v \in \pi^{-1}(\sigma c b)$ such that $\deg(u' - v') \leq 5C$ where $u' = u/\gcd(u,v)$ and $v' = v/\gcd(u,v)$.  The binomial $u' - v'$ is contained in $\C H$ and it reduces $u$.  
 
 Otherwise $\pi u$ is a standard monomial of $\ker \psi$ but $u$ is not a standard monomial of $\ker \phi$.  Then $u$ is reduced by some element of $\Inc F$.
\end{proof}
 
By Proposition \ref{prop:bounded-deg}, $\C H$ is contained in a finite number of $\Inc$ orbits.  Letting $H$ be such a finite set of orbit representatives, $H$ is a $\Inc$ equivariant Gr\"obner basis of $\ker \phi$ with respect to the order $\leq$.  This concludes the proof of Theorem \ref{theo:finGB}.  It is not known if $\ker \phi$ has finite $\Inc$-equivariant Gr\"obner bases monomial orders other than those of the form in Theorem \ref{theo:finGB}.

\subsection{Computing equivariant Gr\"obner bases of toric ideals}

To compute a Gr\"obner basis of $\ker \phi$ from the description of $\phi$ we first assume that $Y = Y'$ for $Y'$ defined in \ref{yprime}.  A Gr\"obner basis of the graph of $\phi$, denoted $\Gamma_{\phi} \subseteq R[Y][X]$, is computed with respect to an elimination order for $X$.  We must prove that the graph has a finite $\Inc$-equivariant Gr\"obner basis with respect to such an elimination order.  Algorithm \ref{alg:truncBuch} then provides a way to compute the Gr\"obner basis.

$\Gamma_{\phi} := \ideal{y - \phi(y) \mid y' \in Y'}$ is itself a $\Sym$-invariant toric ideal.  It is the kernel of the monomial map $\phi':R[Y][X] \to R[X]$ defined by $\phi'(y^Ax^C) = \phi(y^A)x^C$.  Factoring $\phi'$ in the prescribed way produces
\[ R[Y][X] \xrightarrow{\pi'} R[Z][X] \xrightarrow{\psi'} R[X] \]
where $\pi'(y^Ax^C) = \pi(y^A)x^C$ and $\psi'(z^Bx^C) = \psi(z^B)x^C$.

The monoid order $\leq_1$ on $\im \pi$ can be extended to an order $\leq_1'$ on $\im \pi' = (\im \pi)[X]$ that eliminates $X$.  Define $\leq'$ to be the order on $[Y][X]$ such that $y^Ax^C < y^Bx^D$ if $\pi(y^A)x^C <_1' \pi(y^B)x^D$ or $\pi(y^A)x^C = \pi(y^B)x^D$ and $y^A <_2 y^B$.  The restriction of $\leq'$ to $[Y]$ is the hybrid order $\leq$ constructed previously from $\leq_1$ and $\leq_2$.  The order $\leq'$ eliminates $[X]$, and satisfies the hypotheses for Theorem \ref{theo:finGB} so there exists a finite $\Inc$-equivariant Gr\"obner basis $H$ for $\Gamma_\phi$ with respect to $\leq'$.  Using the above algorithm, $H$ can be explicitly computed from the $\Inc$-generators of $\Gamma_\phi$, which are
 \[ \{\sigma y_p - \phi(\sigma y_p) \mid p \in [N],\; \sigma \in \F S_{k_p}\} \]
where $y_1,\ldots,y_N$ are representatives of the $\Sym$-orbits of $Y$.
Then $G = H \cap R[Y]$ is a $\Inc$-equivariant Gr\"obner basis for $\ker(\phi) = \Gamma_{\phi} \cap R[Y]$ for the order $\leq$.

If $Y \neq Y'$, then we can compute an equivariant Gr\"obner basis $G$ for the kernel of $\phi\circ\theta:K[Y'] \to K[X]$ as above.
By Proposition \ref{prop:theta}, $\theta(G)$ is a $\Inc$-equivariant Gr\"obner basis for $\ker \phi$.

\bibliographystyle{alpha}
\bibliography{inftoric}

\newcommand{\etalchar}[1]{$^{#1}$}
\begin{thebibliography}{DEKL15}

\bibitem[AH07]{aschenbrenner2007finite}
Matthias Aschenbrenner and Christopher Hillar.
\newblock Finite generation of symmetric ideals.
\newblock {\em Transactions of the American Mathematical Society},
  359(11):5171--5192, 2007.

\bibitem[AT05]{aoki05:_markov_monte_carlo}
Satoshi Aoki and Akimichi Takemura.
\newblock {M}arkov chain {M}onte {C}arlo exact tests for incomplete two-way
  contingency table.
\newblock {\em Journal of Statistical Computation and Simulation},
  75(10):787--812, 2005.

\bibitem[BD11]{Brouwer09e}
Andries~E. Brouwer and Jan Draisma.
\newblock Equivariant {G}r\"obner bases and the two-factor model.
\newblock {\em Math. Comput.}, 80:1123--1133, 2011.

\bibitem[CEF12]{church2012fi}
Thomas Church, Jordan~S Ellenberg, and Benson Farb.
\newblock Fi-modules: a new approach to stability for sn-representations.
\newblock {\em arXiv preprint arXiv:1204.4533}, 2012.

\bibitem[DEKL15]{draisma2013noetherianity}
Jan Draisma, Rob~H Eggermont, Robert Krone, and Anton Leykin.
\newblock Noetherianity for infinite-dimensional toric varieties.
\newblock {\em Algebra \& Number Theory}, 9(8):1857--1880, 2015.

\bibitem[DO06]{loera06:_markov_bases_of_three_way}
Jes\'{u}s~A. {De Loera} and Shmuel Onn.
\newblock {M}arkov bases of three-way tables are arbitrarily complicated.
\newblock {\em Journal of Symbolic Computation}, 41:173--181, 2006.

\bibitem[DS{\etalchar{+}}98]{diaconis1998algebraic}
Persi Diaconis, Bernd Sturmfels, et~al.
\newblock Algebraic algorithms for sampling from conditional distributions.
\newblock {\em The Annals of statistics}, 26(1):363--397, 1998.

\bibitem[HS12]{hillar2012finite}
Christopher~J. Hillar and Seth Sullivant.
\newblock Finite groebner bases in infinite dimensional polynomial rings and
  applications.
\newblock {\em Advances in Mathematics}, 229(1):1--25, 2012.

\bibitem[KKL14]{KKL:equivariant-markov}
Thomas Kahle, Robert Krone, and Anton Leykin.
\newblock Equivariant lattice generators and markov bases.
\newblock In {\em International Symposium on Symbolic and Algebraic
  Computation}, 2014.

\end{thebibliography}
\end{document}